\documentclass[12pt]{article}
\usepackage{amsmath, amsfonts,amsthm,amssymb,amsbsy,upref,color,graphicx,amscd,hyperref,makeidx,enumerate}
\usepackage{tikz}
\usepackage[active]{srcltx}
\usepackage[latin1]{inputenc}
\usepackage{enumerate}
\usepackage[english]{babel}
\usepackage{newlfont}
\usepackage{amsbsy}
\usepackage[english]{babel}
\usepackage[center]{caption2}
\usepackage{amssymb}
\usepackage{amsmath}
\usepackage{latexsym}
\usepackage{amsthm}
\usepackage{amsthm}
\theoremstyle{plain}
\newtheorem{thm}{Theorem}
\newtheorem{cor}[thm]{Corollary}

\newtheorem{prop}[thm]{Proposition}
\newtheorem{nota}[thm]{Notation}
\newtheorem{rem}[thm]{Remark}
\newtheorem{defin}[thm]{Definition}

\newcommand{\R}{\mathbb{R}}

\newcommand{\N}{\mathbb{N}}
\newcommand{\C}{\mathbb{C}}

\usepackage{mathtools}
\def\multiset#1#2{\ensuremath{\left(\kern-.2em\left(\genfrac{}{}{0pt}{}{#1}{#2}\right)\kern-.2em\right)}}
\usepackage[center]{caption2}

\begin{document}

\title{Families of multiweights and pseudostars}
\author{Agnese Baldisserri \  \ \  \ \ Elena Rubei}
\date{}
\maketitle

{\bf Address of both authors:}
Dipartimento di Matematica e Informatica ``U. Dini'',

viale Morgagni 67/A,
50134  Firenze, Italia

{\bf
E-mail addresses:}
baldisser@math.unifi.it, rubei@math.unifi.it

{\bf Corresponding author:} Elena Rubei

\def\thefootnote{}
\footnotetext{ \hspace*{-0.36cm}

 }

\begin{abstract}
Let  ${\cal T}=(T,w)$ be a weighted  finite tree with leaves $1,..., n$.
For any  $I :=\{i_1,..., i_k \} \subset \{1,...,n\}$,
let $D_I ({\cal T})$ be
the weight of the minimal subtree of $T$ connecting $i_1,..., i_k$;
the  $D_{I} ({\cal T})$ are called 
$k$-weights of  ${\cal T}$. Given  a family 
of  real numbers parametrized by the $k$-subsets of $
\{1,..., n\}$, $\{D_I\}_{I  \in {\{1,...,n\} \choose k}}$,
  we say that a weighted tree ${\cal T}=(T,w)$ with leaves $1,..., n$ realizes the family 
if $D_I({\cal T})=D_I$ for any $ I  $.

In \cite{P-S}  Pachter and Speyer proved that, if $3 \leq  k \leq (n+1)/2$ and
 $\{D_I\}_{I  \in {\{1,...,n\} \choose k}}$ is  a family 
of  positive real numbers, 
  then there exists at most one  positive-weighted essential tree
 ${\cal T}$ with leaves $1,...,n$ that realizes the family (where ``essential'' means that  
there are no vertices of degree $2$). 
We say that a tree $P$  is a pseudostar of kind $(n,k)$ if  the cardinality of the leaf set is $n$ and any edge of $P$ divides the leaf set  
 into two sets such that  at least one of them has cardinality   $ \geq k$.
Here we show that, if $3 \leq  k \leq n-1$ and
 $\{D_I\}_{I  \in {\{1,...,n\} \choose k}}$ is  a family 
of  real numbers realized by some weighted tree, then  there is  exactly one  weighted essential pseudostar  ${\cal P}=(P,w)$ of kind $(n,k)$ 
 with leaves $1,...,n$ and  without internal edges of   weight $0$, that realizes the family; moreover we describe
how any other weighted tree realizing the family can be obtained from ${\cal P}$. 
Finally we examine the range of the total weight of the weighted trees realizing a fixed family.
\end{abstract}

\section{Introduction}

For any graph $G$, let $E(G)$, $V(G)$ and $L(G)$ 
 be respectively the set of the edges,   
the set of the vertices and  the set of the leaves of $G$.
A {\bf weighted graph} ${\cal G}=(G,w)$ is a graph $G$ 
endowed with a function $w: E(G) \rightarrow \R$. 
For any edge $e$, the real number $w(e)$ is called the weight of the edge. If all the weights are nonnegative (respectively positive), 
we say that the graph is {\bf nonnegative-weighted} (respectively {\bf positive-weighted});
if  the weights  of 
the internal edges are nonzero,
 we say that the graph is {\bf internal-nonzero-weighted} and,
if all the weights are nonnegative and the ones of 
the internal edges are positive,
 we say that the graph is {\bf internal-positive-weighted}.
For any finite subgraph $G'$ of  $G$, we define  $w(G')$ to be the sum of the weights of the edges of $G'$. 
In this paper we will deal only with weighted finite trees.

\begin{defin}
Let ${\cal T}=(T,w) $ be a weighted tree. 
For any distinct $ i_1, .....,i_k \in V(T)$,
 we define $ D_{\{i_1,...., i_k\}}({\cal T}) $ to be the weight of the minimal 
subtree containing $i_1,....,i_k$. We call this subtree ``the subtree realizing  $ D_{\{i_1,...., i_k\}}({\cal T}) $''.
 More simply, we denote 
$D_{\{i_1,...., i_k\}}({\cal T})$ by
$D_{i_1,...., i_k}({\cal T})$ for any order of $i_1,..., i_k$.  
We call  the  $ D_{i_1,...., i_k}({\cal T})$ 
the {\bf $k$-weights} of ${\cal T}$ and
 we call  a $k$-weight of ${\cal T}$ for some $k$  a {\bf multiweight}
 of ${\cal T}$.
\end{defin}

%Observe that in the case ${\cal G}$ is a tree,  $ D_{i_1,...., i_k}({\cal G})$ is the sum of the weights of the edges of the  minimal subtree joining $i_1,...,i_k$.

 If $S $ is a subset of $V(T)$,
 the $k$-weights give
a vector in $\mathbb{R}^{  S \choose k}$. This vector is called 
$k${\bf -dissimilarity vector} of $({\cal T}, S)$.
Equivalently,  we can speak 
of the {\bf family of the $k$-weights}  of $({\cal T}, S)$.   

If $S$ is a finite set, $k \in \N$ and $k < \# S$,
we say that a family of  real numbers
 $\{D_{I}\}_{I \in {S \choose k}}$  is {\bf treelike} (respectively p-treelike, nn-treelike, 
inz-treelike, ip-treelike) if there exist a weighted  (respectively 
positive-weighted,  nonnegative-weighted, internal-nonzero-weighted, 
internal-positive-weighted)  tree
 ${\cal T}=(T,w)$ and a subset $S$ 
of the set of its vertices such that $ D_{I}({\cal T}) = D_{I}$  for any 
 $k$-subset $I$ of $ S$.
If in addition 
% the tree is a weighted (respectively  positive-weighted, nonnegative-
%weighted, internal-nonzero-weighted, internal-positive-weighted) tree and 
$S \subset L(T)$,
we say that the family
 is  {\bf l-treelike} (respectively p-l-treelike, nn-l-treelike, inz-l-treelike,
  ip-l-treelike).

A criterion for a metric on a finite set to be nn-l-treelike
%realized by a  nonnegative-weighted tree with leaf set $\{1,..., n\}$ 
was established in  \cite{B}, \cite{SimP}, \cite{Za}:

\begin{thm} \label{Bune} %{\bf (Buneman)} 
Let $\{D_{I}\}_{I \in {\{1,...,n\} \choose 2}}$ be a set of positive real numbers 
satisfying the triangle inequalities.
It is p-treelike (or nn-l-treelike) 
 if and only if, for all distinct $a,b,c,d  \in \{1,...,n\}$,
the maximum of $$\{D_{a,b} + D_{c,d},D_{a,c} + D_{b,d},D_{a,d} + D_{b,c}
 \}$$ is attained at least twice. 
\end{thm}

%In terms of tropical geometry, the theorem above can be formulated %by saying that the set of the $2$-dissimilarity vectors of weighted %trees with $n$ leaves and such that  the internal 
%edges have  negative weights is the tropical Grassmannian $ {\cal %G}_{2,n}$ (see \cite{SS2}).

Also the study of general weighted trees can be interesting and   
in \cite{B-S}, Bandelt and Steel proved a result, analogous to
Theorem \ref{Bune}, for general weighted trees:

\begin{thm} {\bf (Bandelt-Steel)}
 For any set of real numbers $\{D_{I}\}_{I \in {\{1,...,n\} \choose 2}}$,
 there exists a weighted tree ${\cal T}$ with leaves $1,...,n$
such that $ D_{I} ({\cal T})= D_{I}$  for any $I \in {\{1,...,n\} \choose 2}$  if and only  if, for any $a,b,c,d \in  \{1,...,n\}$, at least two among 
 $ D_{a,b} + D_{c,d},\;\;D_{a,c} + D_{b,d},\;\; D_{a,d} + D_{b,c}$
are equal.
\end{thm}

 In \cite{Dress}
some results on graphs  with minimal total weight among the ones realizing a given metric space were established.

For higher $k$ the literature is more recent, see  \cite{B-R},  \cite{BC},
 \cite{H-H-M-S},   \cite{Iri}, \cite{L-Y-P},  \cite{Man}, \cite{P-S}, \cite{Ru1}, \cite{Ru2}.
Three of the most important results for higher $k$ are the following:

\begin{thm} {\bf  (Herrmann, Huber, Moulton, Spillner, \cite{H-H-M-S})}.
If $n \geq 2k$, a family  of positive real numbers 
$ \{D_{I}\}_{I \in {\{1,..., n\} \choose k}}$ is ip-l-treelike 
if and only if the restriction to every $2k$-subset of $\{1,...,n\}$ is 
ip-l-treelike.   
\end{thm}

 \begin{thm} \label{LYP} {\bf (Levy-Yoshida-Pachter,  \cite{L-Y-P})} Let ${\cal T}=(T,w)$ be a positive-weighted tree 
  with $L(T)=[n]$. For any $i ,j \in [n]$, define $$S(i,j) = \sum_{Y \in {[n] -\{i,j\} \choose k-2}} D_{i,j ,Y} ({\cal T}). $$ Then there exists a positive-weighted tree ${\cal T}' =(T',w')$   such that $D_{i,j}({\cal T}')=S(i,j)$  for all $i,j \in [n]$,
  the quartet system of $T'$ is contained in  the quartet system of $T$ and, defined $T_{\leq s}$ the subforest of $T$  whose edge set consists of edges whose removal results in one of the components having size at most $s$, we have 
  $T_{\leq n-k} \cong T'_{\leq n-k}$.
  
   % There is an invertible linear map  between the weights of the edges %in $T_{\leq n-k}$  and  the weights of the edges in $T'_{\leq n-k}$.
  \end{thm}

\begin{thm} \label{PatSp} {\bf (Pachter-Speyer, \cite{P-S})}. Let $ k ,n  \in \mathbb{N}$ with
$3 \leq  k \leq (n+1)/2$.  A positive-weighted tree
 ${\cal T}$ with leaves $1,...,n$ and no vertices of degree 2
is determined by the values $D_I({\cal T})$, where $ I $ varies in  
${\{1,...,n\} \choose k }$.
\end{thm}

Pachter-Speyer Theorem raises naturally some questions.
It is natural to ask what happens if $k>(n+1)/2$: Pachter and Speyer 
showed that in this case  the statement does not hold, but is there a class
of trees such that at most one tree in this class realizes a given family $\{D_I \}_{ I \in {\{1,...,n\} \choose k }}$?

Furthermore we can wonder what happens if we consider general weighted trees instead of positive-weighted ones.

Finally,  if a family $\{D_I \}_{ I \in {\{1,...,n\} \choose k }}$ is realized by more than one (positive-)weighted tree, can we say something about the ones, 
among the weighted trees realizing the family, that have maximum/minimum total weight (if they exist)? 

To state our results we need the following definition.

\begin{defin} Let $k\in \N- \{0\}$.  We say that  a tree $P$ is a {\bf pseudostar} of kind $(n,k)$ if $\ L(P)=n$ and
   any edge  of $P$ divides $L(P)$
 into two sets such that  at least one of them has cardinality   greater than or equal to $ k$.
\end{defin}

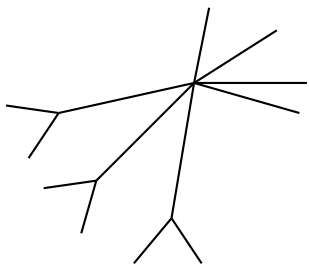
\begin{figure}[h!]
\begin{center}

\begin{tikzpicture}
\draw [thick] (0,0) --(0.2,1);
\draw [thick] (0,0) --(1.1,0.7);
\draw [thick] (0,0) --(1.5,0);
\draw [thick] (0,0) --(1.4,-0.4);
\draw [thick] (0,0) --(-1.8,-0.4);
\draw [thick] (-1.8,-0.4) --(-2.5,-0.3);
\draw [thick] (-1.8,-0.4) --(-2.2,-1);
\draw [thick] (0,0) --(-1.3,-1.3);
\draw [thick] (-1.3,-1.3) --(-2,-1.4);
\draw [thick] (-1.3,-1.3) --(-1.5,-2);
\draw [thick] (0,0) --(-0.3,-1.8);
\draw [thick] (-0.3,-1.8) --(-0.8,-2.4);
\draw [thick] (-0.3,-1.8) --(0.1,-2.4);

\end{tikzpicture}

\caption{ A pseudostar of kind $(10,8)$ \label{pseudostar}}
\end{center}
\end{figure}

We prove that, if $3 \leq k \leq n-1$,  given a l-treelike family of  real numbers,  $\{ D_I\}_{ I  \in {\{1,...,n\} \choose k }}$, 
there exists exactly one internal-nonzero-weighted pseudostar ${\cal P}$ of kind $(n,k) $ with leaves $1,...,n$ and no vertices of degree 2  
such that  $D_I({\cal P})=D_I$ for any $ I  $; it is positive-weighted if the family is 
p-l-treelike. Moreover,  any other tree realizing the family $\{D_I\}_I $ and without vertices of degree $2$  is obtained 
from the pseudostar  by a certain kind of operations we call ``OI operations'' 
and by inserting some internal edges of weight $0$ (see Definition \ref{IO} and Theorem \ref{main}). 
In particular we get that the statement of Pachter-Speyer Theorem holds
also for general weighted trees.

 Finally, in \S4, given a p-l-treelike family  $\{ D_I\}_{ I  \in {\{1,...,n\} \choose k }}$ in the set of positive real numbers,  we examine the range of the total weight of the
trees realizing it and we show that the pseudostar of kind $(n,k) $ realizing it has maximum total weight
(see Theorem \ref{totpositive}); then  we study the analogous problem for l-treelike families in $\R$
(see Theorem \ref{totnonzero}).

%%%%%%%%%%%%%%%%%%%%%%%%%%%%%%%%%%%%%%%%%%%
%%%%%%%%%%%%%%%%%%%%%%%%%%%%%%%%%%%%%%%%%%%

\section{Notation and some remarks}

\begin{nota} \label{notainiziali}

$\bullet$  Let $ \mathbb{R}_{+} =\{x \in \mathbb{R} | \; x >0\}$.

%$\bullet$ We use the symbols $\subset$ and $\subsetneq$  %respectively for the inclusion and the strict inclusion.

$\bullet$  For any $n \in \N $ with $ n \geq 1$, let $[n]= \{1,..., n\}$.

$\bullet$  For any set $S$ and $k \in \mathbb{N}$,  let ${S \choose k}$
be the set of the $k$-subsets of $S$.

$\bullet$ Throughout the paper, the word ``tree'' will denote a finite  tree.

$\bullet$  We say that a vertex of a tree is a {\bf node} if its degree is greater than $2$.
  
$\bullet$  Let $F$ be a leaf of a tree $T$. Let $N$ be the node 
 such that the path $p$ between $N$ and $F$ does not contain any node apart from $N$. We say that $p$ is the {\bf twig} associated to $F$.
We say that an edge is {\bf internal} if it is not an edge of a twig.

%$\bullet$ We say that two weighted trees $ {\cal T} =(T,w) $ and $ {\cal %T}' =(T', w')$ are {\bf equivalent} if each of them can be obtained from %the other by a sequence of operations of the following kind:  replacing %two edges, $e=\{x,y\}$ and $e'= \{y,z\}$, with $y$ vertex  of degree $2$, %with only one edge $e'' =\{x,z\}$ of weight $ w(e) + w(e')$, or the %converse operation. 

$\bullet $ We say that a tree is {\bf essential} if it has no vertices of degree $2$.

$\bullet $ Let $T$ be a tree and let $\{i,j\} \in E(T)$. We say that a tree $T'$  is obtained from $T$ 
by {\bf contracting} $\{i,j\}$ if there exists a map $ \varphi : V(T) \rightarrow V(T') $ 
such that:
\begin{description}\itemsep0.5pt
\item $\varphi (i) = \varphi (j) $, 
\item $ \varphi^{-1} (y) $ is a set with only one element for any $y \neq \varphi (i)$,
\item $E(T')= \big\{ \{\varphi(a), \varphi(b) \}| \;  \{a,b\} \in E(T)\; \mbox{with }  \varphi(a) \neq \varphi(b) \big\}$.
\end{description}
We say also that $T$ is obtained from $T'$ by {\bf inserting} an edge.

$\bullet$ Let $T$ be a tree and let $S$ be a subset of $L(T)$. We denote by $T|_S$ the minimal subtree 
of $T$ whose set of vertices  contains $S$. If ${\cal T}= (T,w)$ is a weighted tree, we denote by 
${\cal T}|_S$  the tree $T|_S$ with the weight induced by $w$.

$\bullet $  Let ${\cal T}=(T,w)$ be a weighted tree. We denote $w(T)$ by $D_{tot} ({\cal T})$ and we call it {\bf total weight}
of ${\cal T}$. 

$\bullet$ Let $n ,k \in \mathbb{N},\;\; n \geq 3$ and $1<k <n$.  Given a family of real numbers $\{D_I\}_{I \in {[n] \choose k }}$, we say that a {\bf weighted tree} ${\cal T}=(T,w)$ with $L(T)=[n]$ {\bf realizes the family $\{D_I\}_{I }$} if 
$D_I ({\cal T}) =D_I$ for any $I \in {[n] \choose k }$.
\end{nota}

\begin{defin} \label{cherries} Let $T$ be a tree.

We say that two leaves  $i$ and $j$ of $T$ are {\bf neighbours}
if in the path from $i$ to $j$ there is only one node; 
furthermore, we say that  $C \subset L(T)$ is a {\bf cherry} if any $i,j \in C$ are neighbours.

We say that a cherry  is {\bf complete}
 if it is not strictly  contained in another cherry.

The {\bf stalk} of a cherry is the unique node in the path 
with endpoints any two elements of the cherry.

Let  $i,j,l,m \in L(T)$. We say that $ \langle
i, j | l, m \rangle $ holds if  in  $T|_{\{i,j,l,m\}}$ we have that $i$ and $j$ are neighbours, 
 $l$ and $m$ are neighbours, and  $i$ and $l$ are not neighbours; in this case we denote by  $\gamma_{i,j,l,m}$ the path 
between the stalk of $\{i,j\}$ and the stalk of $ \{l,m\}$ in $T|_{\{i,j,l,m\}}$. The symbol 
$ \langle i, j | l, m \rangle $  is called {\bf Buneman's index} of $i,j,l,m$. 

\end{defin}

\begin{rem}
(i) A pseudostar of kind $(n,n-1) $ is a star, that is, a tree with only one node.

%(ii)  Let $ r \in \N -\{0\}$. In 
%any $r$-pseudostar that is not a star, all the complete cherries, except  %at most one, have cardinality  less than or equal to $r$.

%(iii) Let $ r \in \N-\{0\}$. By pruning  a cherry in a $r$-pseudostar, we get again a $r$-%pseudostar.

(ii) Let $ k,n \in \N-\{0\}$. If $\frac{n}{2} \geq k $, then every tree 
with $n$ leaves is a pseudostar of kind $(n,k)$, in fact if we divide a set with $n$ elements into two 
parts, at least one of them has cardinality greater than or equal to $ \frac{n}{2}$, which 
is greater than or equal to $ k$. 
\end{rem}

\begin{defin} \label{IO} Let $k,n \in \N -\{0\}$. Let ${\cal T}=(T,w) $ be a weighted tree with $L(T)=[n]$.  Let $e$ be an   edge of $T$  with weight $y$ and dividing $[n]$ into two sets 
 such that each of them has strictly less than $k$ elements.
Contract $e$   and add $y/k$ to the weight of every twig of the tree. 
 We call this operation a $k$-IO operation on ${\cal T}$
and we call the inverse operation a  $k$-OI operation.
\end{defin}

\begin{rem} \label{IOOIw}
It is easy to check that, if ${\cal T}= (T,w) $ and  ${\cal T}'= (T',w') $  are weighted trees with $L(T)= L(T')=[n]$  and ${\cal T}'$ is obtained from ${\cal T} $ by a $k$-IO operation on an edge $e$ of weight  $y$, we have that ${\cal T}$ and ${\cal T}'$ have the same $k$-dissimilarity vector.
Furthermore,  if  ${\cal T} $ is positive-weighted we have that $ D_{tot} ({\cal T}')>  D_{tot}({\cal T})$, precisely 
$$ D_{tot}({\cal T}')= D_{tot}({\cal T}) + \frac{n-k}{k} y  .$$ 
\end{rem}

{\bf Example.} Let  $k=5$, $n=8$. Consider the weighted trees in Figure 
\ref{IOop}, where the labelled vertices are the numbers in bold and the
other numbers are the weights.
The tree on the left 
 is not a pseudostar of kind $(8,5)$ because of the edge $e$;
the tree on the right is obtained from the one on the left by 
a $5$-IO operation on $e$. The $5$-weights of the two trees are the same.

\begin{figure}[h!]

\begin{center}
\begin{tikzpicture} \label{IOop}
\draw [thick] (-0.5,0) --(0.5,0);
\draw [thick] (-0.5,0) --(-1,0.5);
\draw [thick] (-0.5,0) --(-1,-0.5);
\draw [thick] (0.5,0) --(1,0.5);
\draw [thick] (0.5,0) --(1,-0.5);
\draw [thick] (1,0.5) --(1.8,0.7);
\draw [thick] (1,0.5) --(0.8,1.2);
\draw [thick] (1,-0.5) --(1.8,-0.7);
\draw [thick] (1,-0.5) --(0.8,-1.2);
\draw [thick] (-1,0.5) --(-1.8,0.7);
\draw [thick] (-1,0.5) --(-0.8,1.2);
\draw [thick] (-1,-0.5) --(-1.8,-0.7);
\draw [thick] (-1,-0.5) --(-0.8,-1.2);
\node [above] at (0,0) {\small{$e$}};
\node [below] at (0,0) {\scriptsize{$10$}};
\node [left] at (-0.75,0.25)  {\scriptsize{$1$}};
\node [left] at (-0.75,-0.25)  {\scriptsize{$1$}};
\node [right] at (0.75,0.25)  {\scriptsize{$2$}};
\node [right] at (0.75,-0.25)  {\scriptsize{$3$}};
\node [below] at (1.5,0.65) {\scriptsize{$5$}}; %ok
\node [left] at (0.96,0.85)  {\scriptsize{$5$}}; %ok
\node [above] at (1.45,-0.7)  {\scriptsize{$6$}}; %ok
\node [left] at (0.95,-0.85)  {\scriptsize{$5$}}; %ok
\node [below] at (-1.45,0.65)  {\scriptsize{$6$}}; %ok
\node [right] at (-0.95,0.85)  {\scriptsize{$5$}}; %ok
\node [above] at (-1.45,-0.65)  {\scriptsize{$5$}};
\node [right] at (-0.95,-0.85)  {\scriptsize{$6$}};
\node [right] at (1.8,0.7) {${\mathbf{7}}$};
\node [above] at (0.8,1.2) {${\mathbf{8}}$};
\node [right] at (1.8,-0.7) {${\mathbf{6}}$};
\node [below] at (0.8,-1.2) {${\mathbf{5}}$};
\node [left] at (-1.8,0.7) {${\mathbf{2}}$};
\node [above] at (-0.8,1.2) {${\mathbf{1}}$};
\node [left] at (-1.8,-0.7) {${\mathbf{3}}$};
\node [below] at (-0.8,-1.2) {${\mathbf{4}}$};
\draw [thick] (5.5,0) --(5,0.5);
\draw [thick] (5.5,0) --(5,-0.5);
\draw [thick] (5,0.5) --(4.2,0.7);
\draw [thick] (5,0.5) --(5.2,1.2);
\draw [thick] (5,-0.5) --(4.2,-0.7);
\draw [thick] (5,-0.5) --(5.2,-1.2);
\draw [thick] (5.5,0) --(6,0.5);
\draw [thick] (5.5,0) --(6,-0.5);
\draw [thick] (6,0.5) --(6.8,0.7);
\draw [thick] (6,0.5) --(5.8,1.2);
\draw [thick] (6,-0.5) --(6.8,-0.7);
\draw [thick] (6,-0.5) --(5.8,-1.2);

\node [left] at (5.25,0.25)  {\scriptsize{$1$}}; 
\node [left] at (5.25,-0.25)  {\scriptsize{$1$}}; 
\node [below] at (4.55,0.65)  {\scriptsize{$8$}}; 
\node [right] at (5.05,0.85)  {\scriptsize{$7$}}; 
\node [above] at (4.55,-0.65)  {\scriptsize{$7$}}; 
\node [right] at (5.05,-0.85)  {\scriptsize{$8$}}; 
\node [right] at (5.75,0.25)  {\scriptsize{$2$}}; 
\node [right] at (5.75,-0.25)  {\scriptsize{$3$}}; 
\node [below] at (6.4,0.65)  {\scriptsize{$7$}}; 
\node [left] at (5.95,0.85)  {\scriptsize{$7$}}; 
\node [above] at (6.4,-0.65)  {\scriptsize{$8$}}; 
\node [left] at (5.95,-0.85)  {\scriptsize{$7$}}; 
\node [left] at (4.2,-0.7) {${\mathbf{3}}$};
\node [below] at (5.2,-1.2) {${\mathbf{4}}$};
\node [left] at (4.2,0.7) {${\mathbf{1}}$};
\node [above] at (5.2,1.2) {${\mathbf{2}}$};
\node [right] at (6.8,0.7) {${\mathbf{7}}$};
\node [above] at (5.8,1.2) {${\mathbf{8}}$};
\node [right] at (6.8,-0.7) {${\mathbf{6}}$};
\node [below] at (5.8,-1.2) {${\mathbf{5}}$};

\end{tikzpicture}
\caption{ a 5-IO operation}

\end{center}
\end{figure}
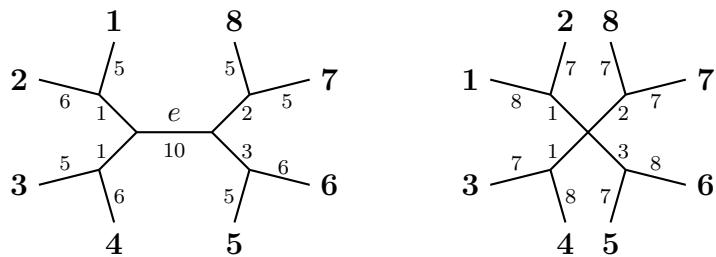

%%%%%%%%%%%%%%%%%%%%%%%%%%%%%%%%%%%%%%%%%%
%%%%%%%%%%%%%%%%%%%%%%%%%%%%%%%%%%%%%%%%%%

\section{Existence and uniqueness of a pseudostar realizing a treelike family}

The following proposition is the core of the proof of Theorem \ref{main}; it shows that  Buneman indices
of  \underline{weighted} pseudostars can be recovered from  their $k$-weights.
The result  was known only for \underline{positive-weighted} trees.

\begin{prop} \label{starbell}
  Let $k,n \in \N$ with $ 2 \leq k \leq n-2$.
Let ${\cal P}=(P,w)$ be a weighted tree with $L(P)=[n]$. 

 \smallskip
 
 1)  Let $i,l \in [n]$.  

(1.1) If $i,l $ are neighbours,  then 
$ D_{i,X} ({\cal P}) - D_{l,X} ({\cal P}) $ does not depend on  $ 
X \in {[n] - \{i,l\} \choose  k-1}$.

(1.2) If ${\cal P}$ is  an internal-nonzero-weighted essential  pseudostar of kind $(n,k)$,
%there are no vertices of degree $2$ and no edges of weight $0$ %in ${\cal P}$ except possibly in its twigs, 
then also the converse is true.

\smallskip

2)  Let $i,j, l,m \in [n]$. 

(2.1) 
If $\langle i,j | l,m \rangle$ holds or $P|_{i,j,l,m}$ is a star, then $$ D_{i,m,R} ({\cal P})  + D_{j,l,R}({\cal P}) = D_{i,l,R} ({\cal P})  + D_{j,m,R}({\cal P}) $$
for any $R \in { [n]- \{i,j,l,m\}  \choose k-2}$.

(2.2) Let $ k \geq 4 $ and
 ${\cal P}$  be  an internal-nonzero-weighted essential pseudostar of kind $(n,k)$.  We have that $\langle i,j | l,m  \rangle$ holds if and only if at least one of the following conditions holds:

(a) $\{i,j\} $ and $\{l,m\}$ are complete cherries in $P$,

(b) there exists $S \in { [n]- \{i,j,l,m\}  \choose k-2}$ such that  
\begin{equation} \label{dis} D_{i,j,S} ({\cal P})  + D_{l,m,S}({\cal P}) \neq  D_{i,l,S} ({\cal P})  + D_{j,m,S}({\cal P}) .\end{equation}
and there exists $R \in { [n]- \{i,j,l,m\}  \choose k-2}$ such that  
\begin{equation} \label{dis2} D_{i,j,R} ({\cal P})  + D_{l,m,R}({\cal P}) \neq  D_{i,m,R} ({\cal P})  + D_{j,l,R}({\cal P}) .\end{equation}

%In the general case,  if $ D_{i,X} ({\cal P}) - D_{j,X} ({\cal P}) $ does not depend on  $ 
%X \in {[n] - \{i,j\} \choose  k-1}$, $ k \geq 4$ and $x, y$ are two nodes on the path from $i$ to $j $ 
%such that there is no node in the path from $x$ to $y$,  then the weight of the path from $x$ to $y$ is %$0$. In particular, if $k \geq 4$ and ${\cal P}$ has no edges of weight $0$ 
%if $ D_{i,X} ({\cal P}) - D_{j,X} ({\cal P}) $ does not depend on  $ X \in {[n] - \{i,j\} \choose  k-1}$,
%then $i$ and $j$ are neighbours.

\end{prop}

\begin{proof}  
(1.1) Obvious.

\medskip
(1.2)  Let ${\cal P}$ be  as in the assumptions
and suppose that 
$ D_{i,X} ({\cal P}) - D_{l,X} ({\cal P}) $ does not depend on  $ 
X \in {[n] - \{i,l\} \choose  k-1}$.
For every $ \delta \in [n]$, let $ \overline{\delta}$ be the 
node on the path from $i$ to $l $   such that 
$$ path(i,l) \cap path(i, \delta) = 
path(i , \overline{\delta}).$$
Suppose, contrary to our claim,  that $i$ and $l$  are not neighbours. 
Therefore, on the path between $i $ and $l $  there are at least two nodes.
 For any $a,b $ nodes in the path between
$i$ and $l$, we say that $ a \leq b$ if and only if  
$path(i, a) \subset path(i, b) $. 
Let $x, y$ be two nodes on the path between $i$ and $l $ such that 
there is no node in the path between $x$ and $y$ apart from $x$ and $y$; thus in 
the path between $x$ and $y$ there is only one edge since $P$ is essential.
Suppose $x < y$, see Figure \ref{tree}. 
We can divide $[n]$ into two disjoint subsets:

$X =\{\delta \in [n] \; |\; \overline{\delta} \leq x\}$,

$Y =\{\delta  \in [n] \; |\; \overline{\delta}\geq y \}$.

Since $P$ is a pseudostar of kind $(n,k)$, then  either $ \# X \geq k $ or $ \# Y \geq k$;
suppose  $ \# X \geq k $ (we argue analogously in the other case); 
let $\gamma_1,..., \gamma_{k-1}$ be distinct elements of $X -\{i\}$ with $ \overline{\gamma_{k-1}}
= x$. 
Up to interchanging the names of $\gamma_1,..,\gamma_{k-2}$  (and 
correspondingly the names of $\overline{\gamma_1},...,\overline{\gamma_{k-2}})$,
we can suppose    
$\overline{\gamma_1} \leq \overline{ \gamma_2} \leq  .......\leq \overline{\gamma_{k-1} }$.
Let $ \eta \in Y -\{l\}$ such that $ \overline{\eta}=y$.

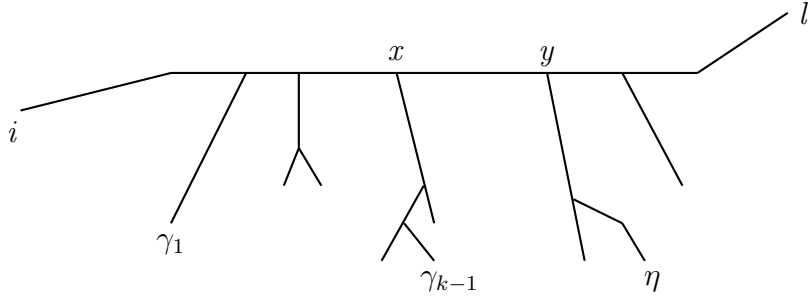
\begin{figure}[h!]
\begin{center}

\begin{tikzpicture}
\draw [thick] (0,0) --(7,0) ; 
\draw [thick] (0,0) --(-2,-0.5);
\draw [thick] (1,0) --(0,-2);
\draw [thick] (1.7,0) --(1.7,-1);
\draw [thick] (1.7,-1) --(2,-1.5);
 \draw [thick] (1.7,-1) --(1.5,-1.5);
\draw [thick] (3,0) --(3.5,-2);
\draw [thick] (3.36,-1.5) --(2.8,-2.5);
\draw [thick] (3.1,-2) --(3.5,-2.5);
\draw [thick] (5,0) --(5.5,-2.5);
\draw [thick] (5.35,-1.68) --(6,-2);
\draw [thick] (6,-2) --(6.3,-2.5);
\draw [thick] (6,0) --(6.8,-1.5);
\draw [thick] (7,0) --(8.2,0.8);

\node [below] at (-2.1,-0.5) {\emph{i}};
\node [below] at (0,-2) {\emph{$\gamma_1$}};
\node [below] at (3.7,-2.5) {\emph{$\gamma_{k-1}$}};
\node [below] at (6.4,-2.5) {\emph{$\eta$}};
\node [right] at (8.2,0.8) {\emph{l}};
\node [below] at (3,0.5) {\emph{x}};
\node [below] at (5,0.5) {\emph{y}};

\end{tikzpicture}

\caption{Neighbours in pseudostars \label{tree}}
\end{center}
\end{figure}

If $k \geq 3$, we have:
$$  D_{i, \gamma_1,...., \gamma_{k-1}}  - D_{l,
 \gamma_1,...., \gamma_{k-1} } 
= w( path (i,\overline{\gamma_{1}} )) -w( path (l,\overline{\gamma_{k-1}} ))  $$
$$  D_{i, \gamma_1, ...., \gamma_{k-2}, \eta}  - D_{l,
 \gamma_1,...., \gamma_{k-2}, \eta } 
= w( path (i,\overline{\gamma_{1}} )) -w( path (l, \overline{\eta} ))  .$$
Since the first members of the equalities above are equal by assumption, we  have that
$$w( path (l,\overline{\gamma_{k-1}} )) = w( path (l, \overline{\eta})),$$ that is 
$$w( path (l,x )) = w( path (l, y)),$$
thus the weight of the edge  $\{x, y\}$ must be $0$, which contradicts the assumption.

If $k=2$,  we have: $$  D_{i, \gamma_1,}  - D_{l,
 \gamma_1} 
= w( path (i,\overline{\gamma_{1}} )) -w( path (l,\overline{\gamma_{1}} )) = w( path (i,x )) -w( path (l,x ))    $$
$$  D_{i, \eta}  - D_{l, \eta } = w( path (i, \overline{\eta} )) -w( path (l, \overline{\eta} )) 
= w( path (i,y )) -w( path (l, y ))  .$$
Since the first members of the equalities above are equal, we must have 
that the weight of the edge $\{x,y\}$ must be $0$, which contradicts the assumption.
\medskip

(2.1) Let $R \in { [n]- \{i,j,l,m\}  \choose k-2}$,
 $A= P|_{i,j,l,m,R}$ and  $A' = P|_{i,j,l,m}$. Suppose $ \langle i,j| l,m \rangle $ holds. Call $x$ the stalk of  the cherry $\{i,j\}$ and $y$ the stalk of the cherry $\{l,m\}$ in $A'$.  Let us denote the set
of the connected components of  $A-A'$ by $ C_{A-A'}$.
For any $H \in C_{A-A'}$, let $v_H$ be the vertex
that is both a vertex of  $H$ and a vertex of  $A'$.
Then $  D_{i, m, R} ({\cal P})$ is equal to  $$  D_{i,m} ({\cal P}) +\sum_{ H \in C_{A-A'}} w(H) +
w \, \left( \bigcup_{\stackrel{ H \in C_{\mbox{\tiny $\!\!A$-$A'$}}\; \mbox{\footnotesize s.t.}
}{ v_H \in  V (path(j, x))}  }  path(v_H,x) \right) +w \, \left(\bigcup_{\stackrel{ H \in C_{\mbox{\tiny $\!\!A$-$A'$}}\; \mbox{\footnotesize s.t.}
}{ v_H \in  V (path(l, y)) } }  path(v_H,y) \right) .$$ 
Analogous formulas hold for $D_{i,l, R} ({\cal P}) , D_{j, m, R} ({\cal P}),
D_{j, l, R} ({\cal P}) $ and we can easily prove our statement. 
If $A'$ is a star,  we argue analogously. 

\medskip

(2.2)
$\Leftarrow$  Obviously (a) implies  $\langle i,j| l,m\rangle$. Suppose (b) holds; then, if $P|_{i,j,l,m}$ were a star or $\langle i,l | j,m \rangle $ or 
$ \langle i, m | j ,l \rangle$ held, from (2.1) we would get a contradiction of the assumptions. Thus  $\langle i,j| l,m\rangle$ holds.

$\Rightarrow$ 
 Let us consider the path between $i$ and $l$. We use the same notation as in (1.2). By assumption $\overline{j} < \overline{m}$.
Let $m' \in [n]$ be such that $\overline{m'}$ is the maximum node strictly less than $ \overline{m}$ and let $j' \in [n]$ be such that $\overline{j'}$ is the minimum node strictly greater than $ \overline{j}$. We could possibly have $j'=m$ and  $m'=j$ or $j'=m'$. In 
 Figure \ref{j'm'} we sketch the situation in case $j' < m'$.

\begin{figure}[h!]
\begin{center}
\begin{tikzpicture}
\draw [thick] (-1.5,0) --(1.5,0) ; 
\draw [thick] (-1.5,0) --(-3,0.2);
\draw [thick] (1.5,0) --(3,0.2);
\draw [thick] (-1.5,0) --(-1.8,-1);
\draw [thick] (1.5,0) --(1.8,-1); 
\draw [thick] (0.8,0) --(1.1,-1);
\draw [thick] (-0.8,0) --(-1.1,-1);
\node [left] at (-3,0.2) {\small $i$};
\node [below] at (-1.8,-1) {\small $j$};
\node [right] at (3,0.2) {\small $l$};
\node [below] at (1.8,-1) {\small $m$};
\node [below] at (1.1,-1) {\small{$m'$}};
\node [below] at (-1.1,-1) {\small{$j'$}};
\end{tikzpicture}
\caption{ $j'$ and $m'$ \label{j'm'}}
\end{center}
\end{figure}

 Since $P$ is a pseudostar of kind $(n,k)$,  we have:
\begin{equation} \label{mm'} 
\# \{x \in [n]| \; \overline{x} \leq \overline{m'} \} \geq k  \;\; \; \mbox { or } \;\; \; 
\# \{x \in [n]| \; \overline{x} \geq \overline{m} \} \geq k \end{equation} 
and 
\begin{equation} \label{jj'}
\# \{x \in [n]| \; \overline{x} \leq \overline{j} \} \geq k  \;\; \; \mbox{ or } \;\; \;  \# \{x \in [n]| \; \overline{x} \geq \overline{j'} \} \geq k.
\end{equation}

\medskip

$\bullet$ First suppose that 
there exists $s \in [n]-\{i,j\}$ such that $ \overline{s}
\leq \overline{j}$ and there exists $t \in [n]-\{l,m\}$ such that $ \overline{t}
\geq \overline{m}$. 
 From (\ref{mm'}) and (\ref{jj'}) we get
 $$\#\{x \in [n]| \; \overline{x} \leq \overline{m'} \} \geq k   \;\; \; \mbox{ or } \;\; \; \#\{x \in [n]| \; \overline{x} \geq \overline{j'} \} \geq k .$$

Suppose  $ \# \{x \in [n]| \; \overline{x} \leq \overline{m'} \} \geq k $ (the other case is analogous). Let $R$ be a $(k-2)$-subset of  
$\{x \in [n] - \{i,j\}| \; \overline{x} \leq \overline{m'} \} $ such that, if $m' \neq j$, then $R$   contains $s$ and $m'$.
Then 
$$ \begin{array}{ll} 
D_{i,j,R} ({\cal P})  - D_{l,j,R}({\cal P})  & = w(path (i, min ( \overline{j}
\cup \overline{R})))-  w(path (l, max ( \overline{j} \cup \overline{R}))) \\
& =
 w(path (i, min ( \overline{R})))-  w(path (l,  \overline{m'} )),
\end{array}$$ 
$$ \begin{array}{ll}
D_{i,m,R} ({\cal P})  - D_{l,m,R}({\cal P}) & =  w(path (i, min ( \overline{m}
\cup \overline{R})))-  w(path (l, max ( \overline{m} \cup \overline{R}))) \\
& =w(path (i, min ( \overline{R})))-  w(path (l, \overline{m} ))
 \end{array}$$
So we get that $D_{i,j,R} ({\cal P})  - D_{l,j,R}({\cal P})- D_{i,m,R} ({\cal P})  + D_{l,m,R}({\cal P}) = -w(\{ \overline{m'}, \overline{m}\})$, which is nonzero by assumption.
Thus  $D_{i,j,R} ({\cal P}) + D_{l,m,R}({\cal P}) \neq  D_{l,j,R}({\cal P})+ D_{i,m,R} ({\cal P})$; hence  (\ref{dis2})  holds.

$\bullet$ Now suppose that   there exists $s \in [n]-\{i,j\}$ such that $ \overline{s}
\leq \overline{j}$ and there does not exist $t \in [n]-\{l,m\}$ such that $ \overline{t}
\geq \overline{m}$ (analogously if the converse holds).  Then $\#\{x \in [n]| \; \overline{x} \leq \overline{m'} \} \geq k$.
By taking   $R$ to be a $(k-2)$-subset of  
$\{x \in [n] - \{i,j\}| \; \overline{x} \leq \overline{m'} \} $ 
such that, if $m' \neq j$, then $R$   contains $s$ and $m'$, we conclude as above that (\ref{dis2}) holds.

$ \bullet $ Finally, if  there does not  exist $s \in [n]-\{i,j\}$ such that $ \overline{s}
\leq \overline{j}$ and there does not exist $t \in [n]-\{l,m\}$ such that $ \overline{t}
\geq \overline{m}$,  then (a) holds.

By considering  the path between $i$ and $m$, we get analogously that either 
(\ref{dis}) holds or (a) holds.
\end{proof}

The following proposition characterizes Buneman indices in terms of
$k$-weights in the case $k=3$.

\begin{prop} \label{k=3}
 Let  $n \geq 5$. 
Let ${\cal P}=(P,w)$ be an essential and internal-nonzero-weighted tree with $L(P)=[n]$ (so it is a pseudostar of kind $(n,3)$). 
Let $i,j,l,m \in [n]$.

  We have that $\langle i,j | l,m  \rangle$ holds if and only if at least one of the following conditions holds:

(a)   there exists $r \in [n]- \{i,j,l,m\}$ such that the inequality 
$$ D_{i,j,l} ({\cal P}) + D_{m,r,l} ({\cal P})\neq D_{i,r,l} ({\cal P}) + D_{m,j,l} ({\cal P})  $$ 
holds and  the inequalities obtained from this  by swapping $i$ with $j$ and/or 
$l$ with $m$ hold.

(b) for any $ r \in [n]- \{i,j,l,m\}$, the following inequalities hold:
 $$ D_{i,j,r} ({\cal P}) + D_{m,l,r} ({\cal P}) \neq D_{i,m, r} ({\cal P}) + D_{j,l,r} ({\cal P}), $$ 
 $$ D_{i,j,r} ({\cal P}) + D_{m,l,r} ({\cal P}) \neq D_{i,l, r} ({\cal P}) + D_{j,m,r} ({\cal P}). $$ 

\end{prop}

\begin{proof}

 $\Rightarrow $  Let $x$ be the stalk of the cherry $\{i,j\}$ in $P|_{i,j,l,m}$ and
let $y$ be the stalk of the cherry $ \{l,m\}$ in $P|_{i,j,l,m}$ . 
Suppose  first that in $V(\gamma_{i,j,l,m})$ there are some nodes of $P$ different form $x$ and $y$;  call $c$   the  node of $P$ in $V(\gamma_{i,j,l,m})- \{x,y\}$ such that $path (x,c) \subset path (x,c') $ for any $c'$ node of $P$ in 
$V(\gamma_{i,j,l,m})- \{x,y\}$  (that is $c$ is  the node  in 
$\gamma_{i,j,l,m}$ ``nearest'' to $x$). 
Let  $r \in [n]$ be such that $path (x,y) \cap path(x,r) = path(x, c)$.
For such an $r$, we have the inequalities in (a), in fact the 
edge   $\{x,c\}$ has nonzero weight by assumption. Thus, if (a) does not
 hold, then there are no nodes of $P$ in $V(\gamma_{i,j,l,m}) - \{x,y\}$, hence $ \gamma_{i,j,l,m}$ is an edge; 
by assumption $w( \gamma_{i,j,l,m})  \neq 0$ and we can prove easily that (b) holds.

$\Leftarrow$  We can easily prove that, if  (a)  holds or  (b) holds, then 
$P|_{i,j,l,m}$ is not a star and $\langle i, m | j,l \rangle$ and 
$\langle i, l | j,m \rangle$ do not hold.
\end{proof}

\begin{cor} \label{Bunem} Let $n, k \in \N$ with $ 3 \leq k \leq n-2$.
Let $\{D_I\}_{I \in {[n] \choose k}}$ be a  family of  real numbers.
The $D_I$ for $I \in { [n] \choose k}$
determine the Buneman's indices of an internal-nonzero-weighted essential pseudostar ${\cal P}=(P,w)$ of kind $(n,k)$  with $L(P)=[n] $  and realizing the family $\{D_I\}_I$.
\end{cor}

In fact, 
by part 1 of Proposition \ref{starbell}, 
the $D_I$ for $I \in { [n] \choose k}$
determine the complete cherries of an internal-nonzero-weighted essential pseudostar ${\cal P}=(P,w)$   of kind $(n,k)$  with $L(P)=[n] $  and realizing the family $\{D_I\}_I$;
so, by part 2 for $k \geq 4$  they determine its Buneman's indices.
For $k=3$ we can use  Proposition \ref{k=3}.

\begin{thm} \label{main} Let $n, k \in \N$ with $ 3 \leq k \leq n-1$.
Let $\{D_I\}_{I \in {[n] \choose k}}$ be a  family of  real numbers.
If it is l-treelike, then there exists exactly  one internal-nonzero-weighted
essential pseudostar ${\cal P} $ of kind $(n,k)$ realizing the family.
Any other weighted essential  tree realizing the family $\{D_I\}_I$  can be obtained from  ${\cal P} $ by $k$-OI operations and by inserting internal edges of weight $0$.

If the family $\{D_I\}_{I \in {[n] \choose k}}$ is  
p-l-treelike, then  ${\cal P} $ is positive-weighted  
and any other positive-weighted essential  tree realizing the family $\{D_I\}_I$  can be obtained from  ${\cal P} $ by $k$-OI operations.
\end{thm}

\begin{proof}
 Let ${\cal T} =(T,w)$ be a weighted  tree with $L(T)=[n]$ and
 realizing the family $\{D_I\}_{I \in {[n] \choose k}}$. Obviously  we can suppose that it is essential. By $k$-IO operations and contracting the internal edges of weight $0$ we can change ${\cal T}$ into an
 internal-nonzero-weighted essential pseudostar ${\cal P}$ of kind $(n,k)$; it realizes the family $\{D_I\}_I $ by Remark \ref{IOOIw}.  If ${\cal T}$  is positive-weighted,
obviously also ${\cal P}$ is positive-weighted.

If $k=n-1$, it is easy to see that there exists at most a weighted essential star with leaves $1,..., n$ realizing the family $\{D_I\}_I$. 

Suppose $k \leq n-2$.
By Corollary \ref{Bunem}, the $D_I$ for $I \in { [n] \choose k}$
determine determine the Buneman's indices, and then 
they determine $P$, in fact it is well known that
 the Buneman's indices of a tree determine the tree (see for instance \cite{Dresslibro}).
 We have to show that the $D_I$  determine also the weights
of the edges of ${\cal P}$.  
The argument is completely analogous to the proof of the theorem in \cite{P-S}; we sketch it for the convenience of the reader. Let $e$ be an edge of $P$ which is not a twig.
Then there exist $i,j,l,m \in [n]$ such that $e = \gamma_{i,j,l,m} $
(see Definition \ref{cherries} for the meaning of $\gamma_{i,j,l,m} $); 
since $P$ is a pseudostar of kind $(n,k)$, there exists $R \in 
{ [n]-\{i,j,l,m\} \choose k-2 }$ such that $e $ is not an edge of $P|_{R}$. Then $$ 2 w(e) = 
D_{i,m,R} ({\cal P})   + D_{j,l,R} ({\cal P})  - D_{i,j,R} ({\cal P}) - D_{l,m,R} ({\cal P})   , $$ 
so $w(e) $ is determined by the $D_I$. For any $I \in {[n] \choose k}$ we have that 
\begin{equation} \label{D_Itwigs}
D_I ({\cal P}) =  \sum_{\stackrel{  \mbox{\footnotesize $e \in E(T|_{I})$     }}{\mbox{\footnotesize $e$ not twig }}} w(e) + \sum_{i \in I} w(e_i),
\end{equation}
where $e_i$ denotes the twig associated to $i$.

So, for any $i,j \in [n]$ and for any $S \in { [n]-\{i,j\} \choose k-1 }$,
 we have:
$$ w(e_i) -w (e_j) = \sum_{l \in (iS) } w(e_l) - \sum_{l \in (jS) } w(e_l) = 
\Big( D_{iS}  - \sum_{\stackrel{  \mbox{\footnotesize $e \in E(T|_{iS})$     }}{\mbox{\footnotesize $e$ not twig }}} w(e)  \Big) - \Big( D_{jS}  - 
 \sum_{\stackrel{  \mbox{\footnotesize $e \in E(T|_{jS})$     }}{\mbox{\footnotesize $e$ not twig }}} w(e) \Big). $$ 
Hence the difference of the weights of the twigs is determined by  the $D_I$.
From the formula (\ref{D_Itwigs})  we get the weight of every twig. 
\end{proof}

\begin{cor}
The statement of Pachter-Speyer Theorem holds
also for general weighted trees.
  
\end{cor}

Finally 
   we point out that the unicity statement of  Theorem \ref{main} in the case of positive-weighted trees has been already proved (even if not stated), in fact it follows from Theorem \ref{LYP}.

%%%%%%%%%%%%%%%%%%%%%%%%%%%%%%%%%%%%%%%
%%%%%%%%%%%%%%%%%%%%%%%%%%%%%%%%%%%%%%%

\section{The range of the total weight}
%Trees realizing a certain $k$-dissimilarity vector with the smallest weight.}

Let 
$\{D_I\}_{I \in {[n] \choose k}}$ be a  p-l-treelike family in $\R_{+}$.
If $2 \leq k \leq (n+1)/2$ we know that  there exists a unique positive-weighted essential tree $\cal{T}=(T,w)$ with $L(\cal{T})=[n]$ and  realizing the family (see Theorem \ref{PatSp}). On the other hand, for $k> (n+1)/2$ this statement no longer holds and, if we call $U$ the set of all positive-weighted trees realizing the family $\{D_I\}_I$, we can wonder which is the range of the total weight of the weighted trees in $U$.
%if there exists a tree $\cal{T'}$ in $U$ such that $w(\cal{T'}) \leq w(\cal{T})$ for all %$\cal{T}$ in $U$.

\begin{thm} \label{totpositive} Let $k,n \in \N$ with $3 \leq k \leq n-1$.
Let $\{D_I\}_{I \in {[n] \choose k}}$ be a p-l-treelike family of positive real numbers and let $U$ be the set of the positive-weighted trees with $[n]$ as set of leaves and
realizing the family $\{D_I\}_I$. Call $\mathcal{P}$  the unique essential pseudostar of kind $(n,k)$  in $U$ (see Theorem \ref{main}). 
The following statements hold:

(i) $$ \displaystyle \sup_{\cal{T} \in U}\,\{\, D_{tot} (\cal{T})\} = D_{tot} (\mathcal{P})$$
and the supremum is attained only by ${\cal P}$;

(ii) if $\#U>1$, then
$$ \inf_{\cal{T} \in U}\,\{\, D_{tot} (\cal{T})\} =D_{tot} (\mathcal{P})-(n-k)\cdot m$$
where $m$ is the minimum among the weights of the twigs of $\mathcal{P}$; the infimum is not attained. 
\end{thm}

\begin{proof}
(i) Let $\cal{T}=(T,w)$ be a weighted tree in $U$. Without changing
the dissimilarity family and the total weight we can suppose that it is essential. By using several $k$-IO  operations, we can transform it into a
pseudostar of kind $(n,k)$. By Remark \ref{IOOIw} the dissimilarity family does not change, so, by Theorem \ref{main}, 
the pseudostar of kind $(n,k)$  we have obtained 
 must be the unique essential pseudostar of kind $(n,k)$  in $U$, that is $\mathcal{P}$.  By Remark \ref{IOOIw}
 we have that  $D_{tot} (\cal{T}) \leq D_{tot}(\mathcal{P})$; furthermore, if $\cal{T}$ is different from $\mathcal{P}$, then $D_{tot}(\cal{T}) < D_{tot}(\mathcal{P})$.

(ii) Suppose $\#U>1$. Then we can make a $k$-OI operation 
 on $\mathcal{P}$: we add an edge of weight $kx$, where $x<m$, in such a way that 
the edge divides the tree in two  trees each with less than $k$ leaves, and we subtract $x$ from the weight of every twig of $\mathcal{P}$. Let ${\cal T}$ be the tree we have obtained. 
We have 
$$D_{tot} ({\cal T})=D_{tot} (\mathcal{P})+k \cdot x-n \cdot x=D_{tot} (\mathcal{P})-(n-k) \cdot x$$
Obviously,  
the limit of  $D_{tot} ({\cal T})$, as $x$ approaches $m$, is $D_{tot} (\mathcal{P})-(n-k)\cdot m$.

Finally, let ${\cal A} \in U$. Without changing
the dissimilarity family and the total weight we can suppose that it is essential. 
 We can transform ${\cal A}$ into $\mathcal{P}$ by several $k$-IO operations, contracting edges with weights $y_1,y_2,...y_s$ and
adding $\frac{y_1+y_2+...+y_s}{k} $ to the weight of every twig. Then we get:
\begin{equation} \label{wawp}
D_{tot} (\mathcal{P})=D_{tot}({\cal A})+ \frac{n-k}{k}(y_1+y_2+...+y_s). 
\end{equation} 
Furthermore, since, to obtain $\mathcal{P}$ from ${\cal A}$, we have added  $\frac{y_1+y_2+...+y_s}{k} $ to the weight of every twig, we have that 
\begin{equation} \label{m}
m>\frac{y_1+y_2+...+y_s}{k}.
\end{equation}
Thus, from (\ref{wawp}) and (\ref{m}) we get:
$$ D_{tot} ({\cal A})=D_{tot} (\mathcal{P})- \frac{n-k}{k}(y_1+y_2+...+y_s)> D_{tot} (\mathcal{P}) - (n-k) \cdot m.$$
\end{proof}

The following theorem answers  the analogous problem for general weighted trees.

\begin{thm}  \label{totnonzero} Let $k,n \in \N$ with $3 \leq k \leq n-1$.
Let $\{D_I\}_{I \in {[n] \choose k}}$ be a l-treelike family of real numbers and let $U$ be the set of weighted trees with $[n]$ as set of leaves and
realizing the family $\{D_I\}_I$. 

(i)  If in $U$ there are only  weighted pseudostars of kind $(n,k)$  (for instance if $k \leq \frac{n}{2} $), then $ D_{tot}({\cal P})= D_{tot}({\cal P}') $ for any 
${\cal P}, {\cal P}' \in U$; in particular 
$$ \inf_{\cal{P} \in U}\,\{\, D_{tot}(\cal{P})\} =  \sup_{\cal{P} \in U}\,\{\, D_{tot} (\cal{P})\}.$$

(ii) If in $U$ there are weighted trees that are not  pseudostars of kind $(n,k)$, then
$$ \inf_{\cal{T} \in U}\,\{\, D_{tot} (\cal{T})\} = - \infty, \;\;\;\; \;\; \;\; \;\; \sup_{\cal{T} \in U}\,\{\, D_{tot} (\cal{T})\}= + \infty.$$
\end{thm}

\begin{proof}
(i) Let ${\cal P}$ and ${\cal P}'$ be in $U$. Denote by $\hat{{\cal P}}$ and $ \hat{{\cal P}'}$ the weighted trees 
obtained respectively from ${\cal P}$ and ${\cal P}'$ by contracting the
internal  edges of weight $0$. Let $\overline{{\cal P}}$ and $ \overline{{\cal P}'}$ be the weighted essential trees equivalent respectively to $\hat{{\cal P}}$ and $ \hat{{\cal P}'}$.
Obviously both  $\overline{{\cal P}}$ and $ \overline{{\cal P}'}$ pseudostars of kind $(n,k)$  and realize the family 
$\{D_I\}_I$; so,  by Theorem \ref{main},  they are equal. Thus  $D_{tot} (\overline{\cal P}) = D_{tot} (\overline{{\cal P}'})$,
therefore $D_{tot} ({\cal P}) = D_{tot} ({\cal P}')$.

(ii) Let ${\cal T}=(T,w) $ be an element of  $U$ that is  not a pseudostar
of kind $(n,k)$; then there is an edge $\overline{e}$ dividing 
the tree into two trees such that  each of them has less than $k$ leaves.
Let $z \in \R$. We define on $T$ a new weight $w'$ as follows:
$$w' ( \overline{e})  := w (\overline{e}) + z ;$$
for every twig $t$, 
$$ w' (t) := w(t) - \frac{1}{k} z  ;$$ 
for any edge $e$ different from $ \overline{e}$ and not contained in a twig, we define $w'(e) = w(e)$.

Let ${\cal T}' = (T,w')$. We have that $D_I ({\cal T}') = D_I$ for any $I \in {[n] \choose k}$, so ${\cal T}' \in U$. Furthermore   $$D_{tot} ({\cal T}')=
D_{tot}({\cal T}) +z - \frac{n}{k} z=   D_{tot}({\cal T}) - \frac{n-k}{k} z.$$ 
Hence  $\lim_{z \rightarrow - \infty}D_{tot} ({\cal T}') =+ \infty$, while
$\lim_{z \rightarrow + \infty}D_{tot} ({\cal T}') =- \infty$.
\end{proof}

%{\bf Acknowledgemnts.}
%This work was supportes by the National Group for Algebraic and %Geometric Structures, and their  Applications (GNSAGA-INdAM). 
%The first author was supported by Ente Cassa di Risparmio di %Firenze.

\bigskip

\end{document}